\documentclass[12pt,reqno]{amsart}
\usepackage{fullpage}
\usepackage{times}
\usepackage{amsmath,amssymb,amsthm,url}
\usepackage[utf8]{inputenc}
\usepackage[english]{babel}
\usepackage{comment}
\usepackage{bbm}
\usepackage{enumerate}
\usepackage{bm}
\usepackage{graphicx}
\usepackage{mathrsfs}
\usepackage{float}
\usepackage{subcaption}
\usepackage{enumitem}
\usepackage{soul}

\usepackage[colorlinks=true, pdfstartview=FitH, linkcolor=blue, citecolor=blue, urlcolor=blue]{hyperref}
\usepackage{tabstackengine}
\stackMath

\newtheorem{thm}{Theorem}[section]
\newtheorem{lem}[thm]{Lemma}

\newtheorem{cor}[thm]{Corollary}

\theoremstyle{definition}

\newcommand{\R}{\mathbb{R}}

\newcommand{\E}{\mathbb{E}}

\newcommand{\set}[1]{\left\{#1\right\}}

\title{Any two-coloring of the plane contains monochromatic 3-term arithmetic progressions}

\author{Gabriel Currier}
\address{Department of Mathematics \\ University of British Columbia \\ Vancouver  V6T 1Z2 \\ Canada}
\email{currierg@math.ubc.ca}
\author{Kenneth Moore}
\address{Department of Mathematics \\ University of British Columbia \\ Vancouver  V6T 1Z2 \\ Canada}
\email{kjmoore@math.ubc.ca}
\author{Chi Hoi Yip}
\address{Department of Mathematics \\ University of British Columbia \\ Vancouver  V6T 1Z2 \\ Canada}
\email{kyleyip@math.ubc.ca}
\begin{document}

\begin{abstract}
A conjecture of Erd\H{o}s, Graham, Montgomery, Rothschild, Spencer, and Straus states that, with the exception of equilateral triangles, any two-coloring of the plane will have a monochromatic congruent copy of every three-point configuration. This conjecture is known only for special classes of configurations. In this manuscript, we confirm one of the most natural open cases; that is, every two-coloring of the plane admits a monochromatic congruent copy of any $3$-term arithmetic progression.
\end{abstract}

\maketitle

\section{Introduction}

Let $\E^n$ denote the $n$-dimensional Euclidean space, that is, $\R^n$ equipped with the Euclidean norm. Suppose that we have a coloring of $\E^n$ using $r$ colors. The field of Euclidean Ramsey theory concerns itself with the question of what types of configurations (monochromatic, rainbow, etc.) must be present in such a coloring. One of the most commonly studied configurations is denoted $\ell_m$, and consists of $m$ collinear points with consecutive points of distance one apart. In other words, $\ell_m$ is an $m$-term arithmetic progression with common difference $1$. In this note, we prove the following.

\begin{thm}\label{thm:ell3}
In any two-coloring of $\E^2$, there exists a monochromatic congruent copy of $\ell_3$.
\end{thm}

Thus, by scaling, there naturally exists a monochromatic $3$-term arithmetic progression with any common difference. The classical question in this area, known as the Hadwiger-Nelson (HN) problem, is one of the most famous open problems in combinatorics. The HN problem, first discussed by Nelson (not in print) in $1950$, asks how many colors one would need to color $\E^2$ so that there is no monochromatic copy of $\ell_2$; i.e. two points of unit distance apart. This quantity is known sometimes as the \emph{chromatic number} $\chi(\E^2)$ of the plane. It was known that the answer is between $4$ and $7$ for a long time, and a $2018$ breakthrough by de Grey \cite{dG18} showed that one needs at least $5$ colors. In general, it is known that $(1.239+o(1))^n \leq \chi(\E^n)\leq (3+o(1))^n$ as $n \to \infty$ \cite[Section 11.1]{GOT18}.

After the introduction of the HN problem, the area was further developed by Erd{\H{o}}s, Graham, Montgomery, Rothschild, Spencer, and Straus in a series of papers \cite{E73, E75a, E75b}. In these papers, they ask if, for any non-equilateral three-point configuration $K$, there must be a monochromatic congruent copy of $K$ in any $2$-coloring of $\E^2$. The conjecture was confirmed when the coloring is assumed to be polygonal \cite{JKSV09}, but it is still widely open in general. As noted in \cite[Section 6.3]{BMP}, Theorem \ref{thm:ell3} gives perhaps the most natural open case of this conjecture. This question was raised again in the concluding remarks of a very recent paper of F\"uhrer and T\'oth \cite[Page 12]{FT24}.

To discuss further known results, we introduce some standard notation. If we have configurations $K_1,\dots,K_r$ in $\E^n$, we say that $\E^n \rightarrow (K_1,\dots,K_r)$ if, for any coloring of $\E^n$ with $r$ colors, there exists a monochromatic (congruent) copy of $K_i$ in color $i$, for some $i$. If there exists a coloring where this does not hold, we say $\E^n \not \to (K_1,\dots,K_r)$. For simplicity, if $K_i = K$ for all $i$ and $\E^n \rightarrow (K_1,\dots,K_r)$ or $\E^n \not \rightarrow (K_1,\dots,K_r)$, we say simply $\E^n \xrightarrow[]{r} K$ (resp. $\E^n \not \xrightarrow[]{r} K$). 

Using the above terminology, our Theorem \ref{thm:ell3} says that $\E^2 \xrightarrow[]{2} \ell_3$. The question of for which $n,r,s_1,\dots,s_r$ we have $\E^n \xrightarrow[]{} (\ell_{s_1},\dots,\ell_{s_r})$ also has a rich history, so we collect here the known results. The two most relevant results are $\E^2 \not \xrightarrow[]{3} \ell_3$ and $\E^3 \xrightarrow[]{2} \ell_3$. The first of these was shown by Graham and Tressler \cite{GT11} using a simple hexagonal grid construction. On the other hand, Erd{\H{o}}s et. al. \cite[Theorem 8]{E73} proved that $\E^3 \xrightarrow[]{2} T$ for any triangle\footnote{Throughout, degenerate triangles (that is, three collinear points) are also regarded as triangles.} $T$; in particular, $\E^3 \xrightarrow[]{2} \ell_3$. The following are other relevant results in the area:

\begin{itemize}
    \item $\E^2 \to (\ell_2, K)$ for any $K$ with $4$ points (Juh\'{a}sz \cite{J79})
    \item $\E^2 \to (\ell_2, \ell_5)$ (Tsaturian \cite{T17})
    \item There is a set $K$ with $8$ points, such that $\E^2 \not \to (\ell_2, K)$ (Csizmadia and T\'{o}th \cite{CT94})
    \item $\E^3 \to (\ell_2, \ell_6)$ (Arman and Tsaturian \cite{AT18})
    \item $\E^n \not \to (\ell_2, \ell_{2^{cn}})$ for some constant $c>0$ (Conlon and Fox \cite{CF19})
    \item $\E^n \not \to (\ell_3, \ell_{10^{50}})$ (Conlon and Wu \cite{CW23})
    \item Very recently, $\E^n \not \to (\ell_3, \ell_{1177})$ (F\"uhrer and T\'oth \cite{FT24}) 
    \item $\E^n \not \xrightarrow[]{2} \ell_6$ (Erd{\H{o}}s et. al. \cite[Theorem 12]{E73})
\end{itemize}

An \emph{$(a,b,c)$ triangle} is a triangle with side lengths $a,b,c$. The following theorem is due to Erd{\H{o}}s et. al. \cite[Theorem 1]{E75b}. 
\begin{thm}[Erd{\H{o}}s et. al.]\label{thm:Erdos}
Let $n \geq 2$. A given $2$-coloring of $\E^n$ admits a monochromatic $(a,b,c)$ triangle if and only if it admits a monochromatic equilateral triangle of side $a$, $b$, or $c$.     
\end{thm}

Using Theorem~\ref{thm:Erdos} along with Theorem~\ref{thm:ell3}, we get the following corollary. 
\begin{cor}
If $n \geq 2$, then $\E^n \xrightarrow[]{2} T$ for an $(\alpha, 2\alpha, x\alpha)$ triangle $T$ for any $\alpha>0$ and $x\in [1,3]$.
\end{cor}
\begin{proof}
Note that $\ell_3$ is a $(1,1,2)$ triangle. Thus, for any $\alpha>0$, Theorem~\ref{thm:ell3} implies any $2$-coloring of $\E^n$ admits a monochromatic $(\alpha,\alpha,2\alpha)$ triangle by scaling. Theorem~\ref{thm:Erdos} then implies that any $2$-coloring of $\E^n$ admits a monochromatic equilateral triangle of side either $\alpha$ or $2\alpha$. In both cases, the corollary follows from Theorem~\ref{thm:Erdos}.
\end{proof}

This verifies another interesting case of the aforementioned conjecture of Erd\H os et. al. from \cite{E75b}. We refer to \cite{E75b, S76} and \cite[Theorem 11.1.4 (a)]{GOT18} for a collection of known families of triangles $T$ such that $\E^2 \xrightarrow[]{2} T$. In particular,  Erd{\H{o}}s et. al. \cite{E75b} showed that $\E^2 \xrightarrow[]{2} T$ if $T$ has a ratio between two sides equal to $2\sin (\theta/2)$ with $\theta \in \{30^\circ, 72^\circ, 90^\circ, 120^\circ \}$. Our result handles the case that $\theta=180^\circ$.

Erd{\H{o}}s et. al. \cite[Theorem 3]{E75a} showed that $\E^2 \to (K,K')$ for any configuration $K,K'$ of size $2$ and $3$, respectively. Szlam \cite{S01} strengthened this result and showed that every red-blue coloring of the plane in which no two red points are at distance 1 contains at least one blue \emph{translate} of every three-point configuration. In fact, he showed a much stronger result; see \cite[Theorem 1.3]{CF19}. Theorem~\ref{thm:ell3} implies the following, which can be viewed as an analogue of Szlam's theorem.

\begin{cor}
Every red-blue coloring of $\E^n$ contains either a red copy of $\ell_3$, or a blue translate of every 2-point configuration. 
\end{cor}
\begin{proof}
We proceed similarly as in Szlam \cite{S01} with the help of Theorem~\ref{thm:ell3}. Suppose that there is a red/blue coloring of $\mathbb{E}^n$ with no blue translate of some $K=\{a,b\} \subset \mathbb{E}^n$. It follows that for each $p \in \mathbb{E}^n$, $p+a$ or $p+b$ is red.
Now we color the points of $\mathbb{E}^n$ in green and yellow using the following rule: a point $p$ is colored green if $p+a$ is red, and otherwise $p$ is colored yellow. Since $\E^n \to (\ell_3, \ell_3)$, there are 3 collinear points $v_1, v_2, v_3$ with distance one apart that are all green or all yellow. If $v_1,v_2,v_3$ are all green, then $v_1+a, v_2+a, v_3+a$ are all red; if $v_1,v_2,v_3$ are all yellow, then $v_1+b, v_2+b, v_3+b$ are all red. In both cases, we get a red copy of $\ell_3$ and we are done.
\end{proof}

\section{Proof of theorem \ref{thm:ell3}}
\label{sec:proof}
If we assume a $2$-coloring contains no monochromatic $\ell_3$, then the aforementioned result of Erd{\H{o}}s et. al. (Theorem~\ref{thm:Erdos}) implies that it contains no monochromatic equilateral triangles of side-length $1$ or $2$ as well. Thus, it will be useful to construct a point set that contains many $\ell_3$'s, as well as equilateral triangles of side-length $1$ and $2$. The points we use here are all in the following coordinates:
\begin{equation}
\label{eq:theform}
     [a,b,c,d] := \left(\frac{a\sqrt{3} + b\sqrt{11}}{12}, \frac{c + d\sqrt{33}}{12}\right),
\end{equation}
where $a,b,c,d$ are integers. Distances between such points can be checked with integers only, preventing decimal uncertainty in our computations. We will provide the point set coordinates in Appendix \ref{app:pointsets}, and online here \cite{Resources3AP}. It is worth noting that points of the form in \eqref{eq:theform} are very similar to those used to prove that the chromatic number of $\E^2$ is at least $5$ in \cite{dG18} and \cite{EI20}. 

Our proof will proceed in two steps. We will first show that in any two-coloring of $\E^2$ avoiding monochromatic $\ell_3$'s, any unit equilateral triangle colored blue-blue-red will have a blue centroid (and symmetrically, a red-red-blue unit equilateral triangle will have a red centroid). We will use this to color all points in the triangular grid, and use that coloring to derive a contradiction. We proceed now with the first part of proof.

\begin{lem}
\label{lem:computationallemma}
    If $\E^2$ is two-colored without a monochromatic $\ell_3$, any unit equilateral triangle colored red-blue-blue has a blue centroid.
\end{lem}
\begin{proof}
    We will give two proofs of the lemma, one of which is easily verified computationally, and the other of which can be done by hand. For the first approach, we use the point-set depicted in Figure \ref{fig:lemma1graph}, where a unit-length triangle is colored in opposition to the claim of the lemma. 
    
    It can then be checked that there is no way to complete the coloring of the points without creating a monochromatic copy of $\ell_3$ or equilateral triangle of side-length $1$ or $2$. To be certain, we used two different computational methods to verify this; a basic depth-first-search coloring algorithm, and integer programming. The code for both of these can be found here \cite{Resources3AP}.
    \begin{figure}[H]
        \centering
        \includegraphics[scale = 0.45]{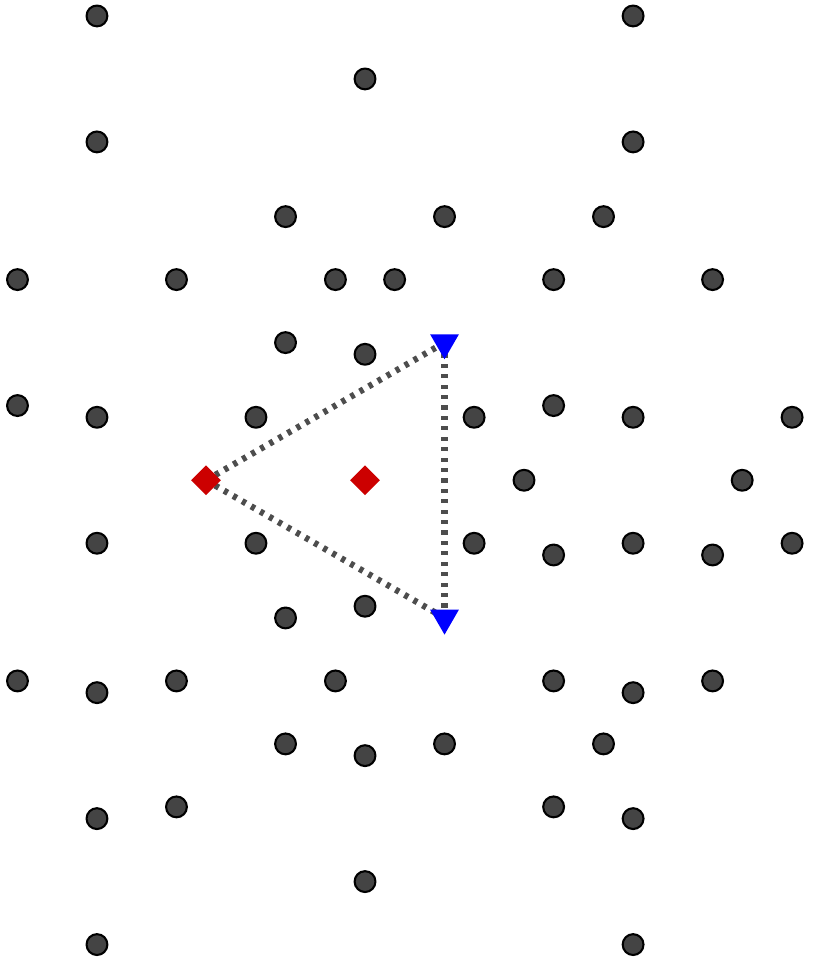}
        \caption{A set of 56 points admitting no coloring compatible with the initial coloring of the unit triangle and its center}
        \label{fig:lemma1graph}
    \end{figure}

   Next, we describe our second argument which can be verified by hand. For this argument, we will use the points described in Figure \ref{fig:lemma1cases}, which have the following coordinates: 
   \begin{equation*}
   \begin{matrix}
    p_1 = [-4,0,0,0], & p_2 = [0,0,0,0], & p_3 = [2,0,-6,0], & p_4 = [2,0,6,0],
    \\ 
    q_1 = [-1,-3,3,-1], & q_2 = [-1,-3, -3, 1], & q_3 = [2,0,0,2], & q_3' = [2,0,0,-2],
    \\
    q_4 = [-3,-3,-3,-1], & q_5 = [-3,-3,3,1]. & &
   \end{matrix}
   \end{equation*}
   We will break into several cases depending on the colors of the points $\set{q_1,...\, ,q_5}$.
    \begin{figure}[H]
        \centering
        \includegraphics[scale = 0.58]{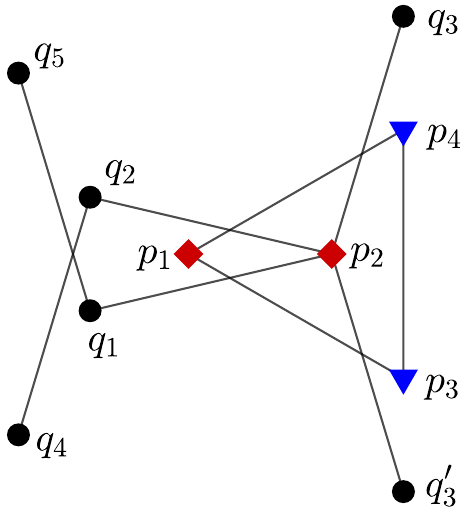}
        \caption{The base points needed to verify the lemma}
        \label{fig:lemma1cases}
    \end{figure}
    The constructions follow by iteratively choosing a select pair of points with the same color that are distance $1$ or $2$ apart, and performing one of the following steps:
    \begin{itemize}
        \item adding a third point to form an $\ell_3$ (as a midpoint if the points are distance $2$ apart), or
        \item adding a third point to create an equilateral triangle,
    \end{itemize} 
    and coloring the new point the opposite color. After a number of steps, we encounter a point that cannot be colored either red or blue. 
   \begin{figure}[H]
        \centering
        \includegraphics[scale = 0.55]{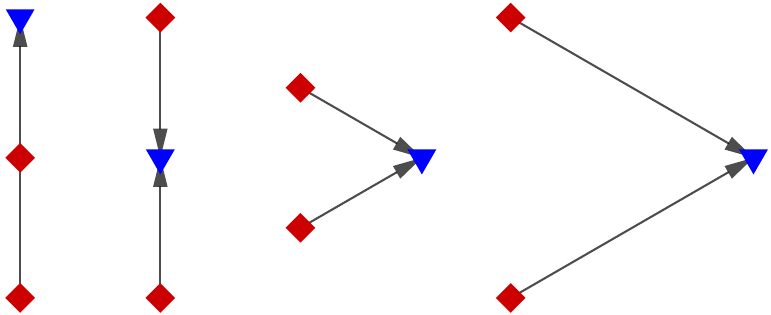}
        \caption{The color implication steps}
        \label{fig:basicmoves}
    \end{figure}
    Because all cases are identical in concept, we treat only the case when $q_1$ is red and $q_2$ is blue here, and relegate the remaining cases to Appendix \ref{app:lemmacases}. In the following graph, we have added 18 points to Figure \ref{fig:lemma1cases}, labeled $\set{s_1,...\, , s_{18}}$, and found that $s_{18}$ cannot be colored red or blue. To verify the claim in this case, one should visit each vertex $s_k$ and observe that its color is implied by vertices in $\set{p_1,...\, ,p_4}$, $\set{q_1,...\, ,q_5,q_3'}$, and $\set{s_1,...\, ,s_{k-1}}$. Note that $p_1$ is not used in this particular case.
    \end{proof}
    \begin{figure}[H]
        \centering
        \includegraphics[scale = 0.54]{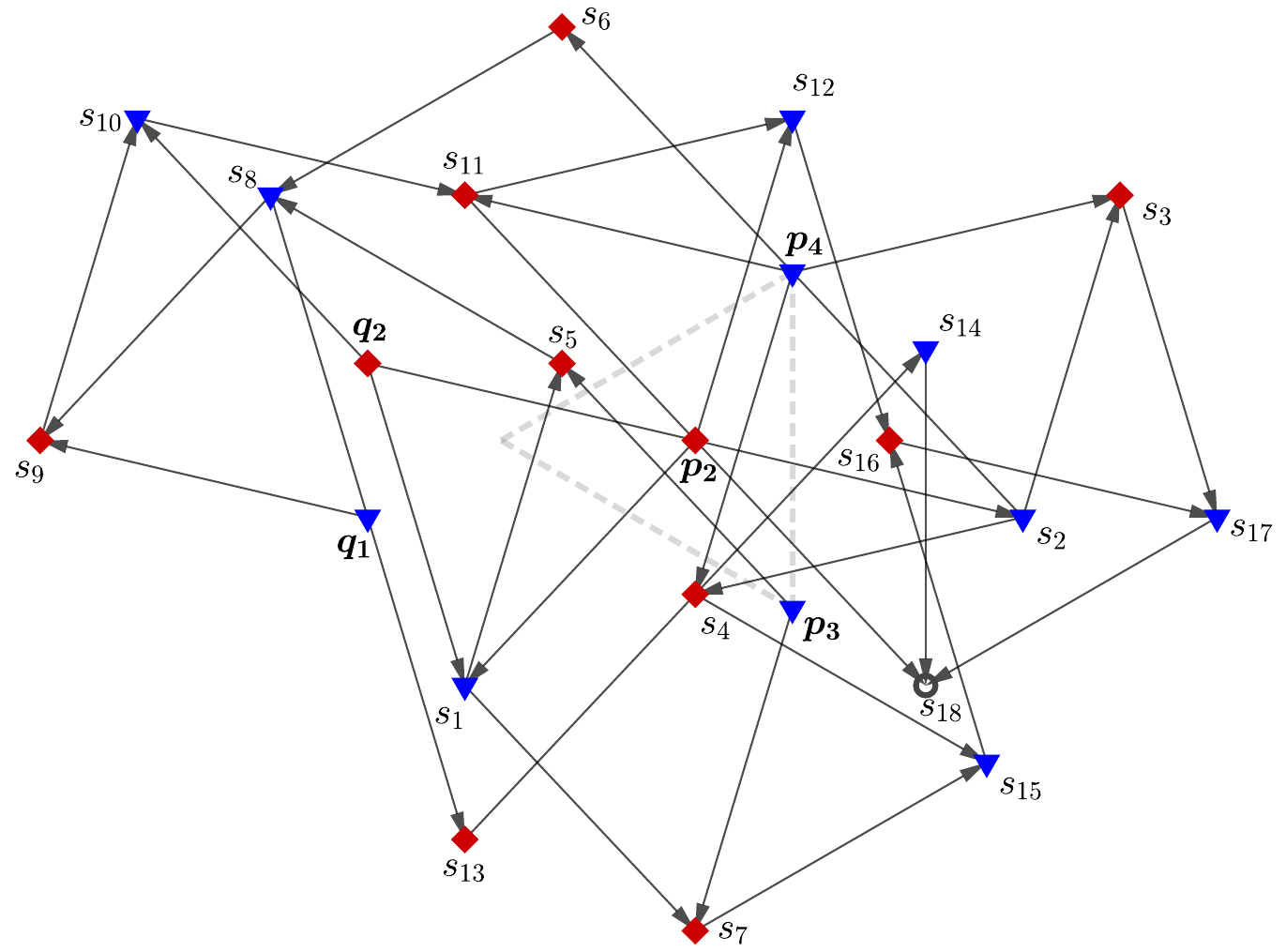}
        \caption{Case 1: $q_1$ is red and $q_2$ is blue (or vice versa)}
        \label{fig:lemma1case1}
    \end{figure}
Now, we can proceed with the second part of the proof of Theorem \ref{thm:ell3}.
\begin{lem}
\label{lem:fixedgrid}
    If $\E^2$ is two-colored without a monochromatic $\ell_3$, a $\frac{1}{\sqrt{3}}$ scaled hexagonal grid has only one valid coloring up to isometry.
\end{lem}

\begin{proof}
    Without loss of generality, we can suppose there is a unit triangle $\set{a_0,b_0,c_0}$ colored red-blue-blue in the grid, and label a fourth point $x_0$ as shown in the left half of Figure \ref{fig:lemma2graph}. $x_0$ could be red or blue, so let us assume first that it is red.
    
    \begin{figure}[H]
        \centering
        \includegraphics[scale = 0.52, trim = 0 35 0 35, clip]{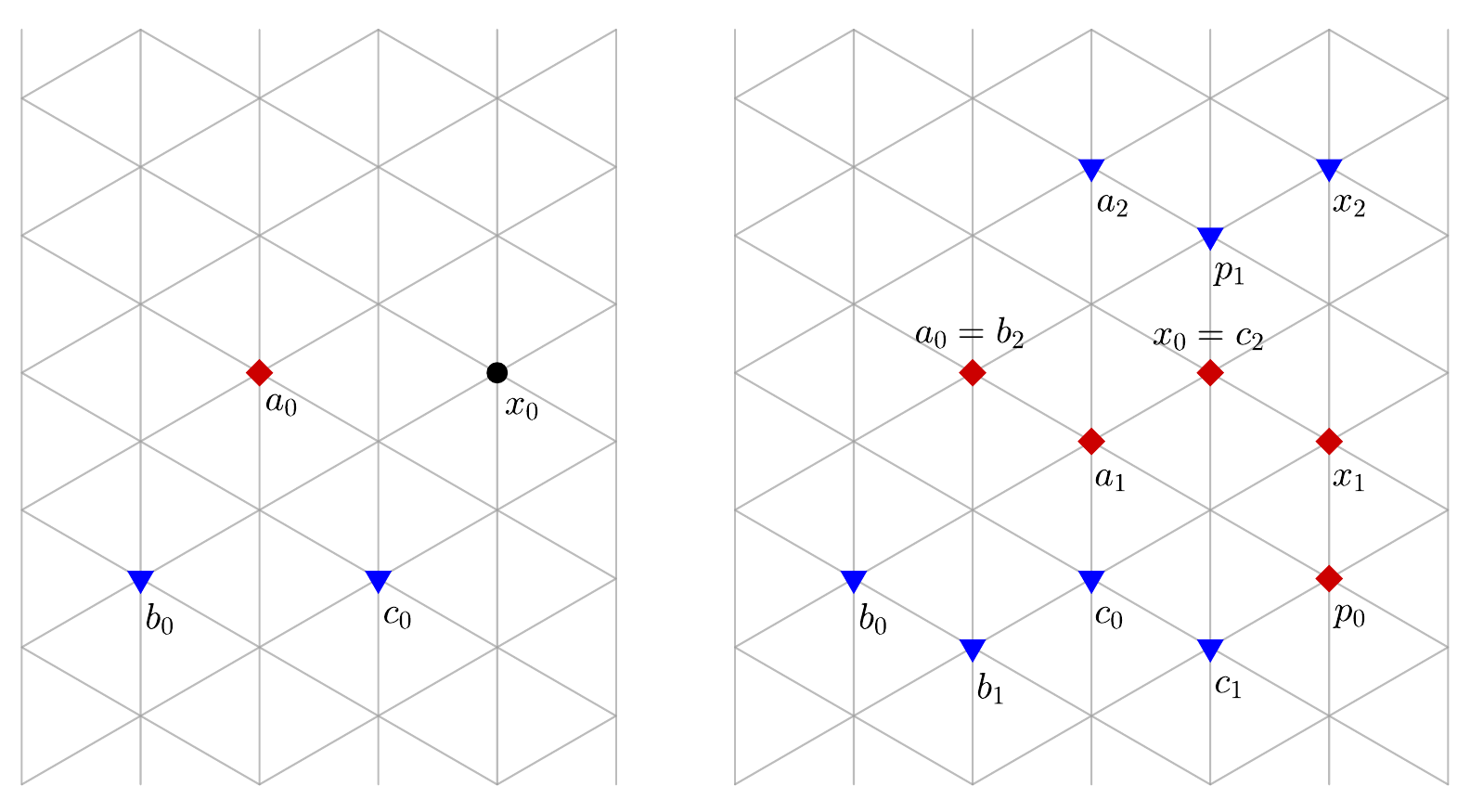}
        \caption{Populating the $\frac{1}{\sqrt{3}}$ scaled hexagonal grid with color}
        \label{fig:lemma2graph}
    \end{figure}
    The sequence of color implications leading to the picture in the second half of Figure \ref{fig:lemma2graph} is 
    \begin{align*}
        b_0, c_0 \implies p_0 \text{ \&  } b_1{}^*, 
        \ \ \
        a_0,x_0 &\implies a_1{}^* \text{ \& } a_2, 
        \ \ \
        c_0,a_1,b_1 \implies c_1{}^*, 
        \\ 
        p_0,x_0 &\implies x_1{}^*,
        \ \ \
        a_1,x_1 \implies p_1,\ \text{ and }
        \ \
        a_2,x_0,p_1 \implies x_2{}^*.
     \end{align*}
    The steps with ${}^*$ are applications of Lemma \ref{lem:computationallemma}. This process can now be repeated using the 4-tuples $\set{a_1,b_1,c_1,x_1}$, and $\set{a_2,b_2,c_2,x_2}$, and in the other direction with $\set{x_0,x_2,a_2,a_0}$ in order to inductively color all vertices in the grid. 

    The case that $x_0$ is blue can be reduced to the case where it is red using the following steps.
    \begin{figure}[H]
        \centering
        \includegraphics[scale = 0.54, trim = 0 0 0 0, clip]{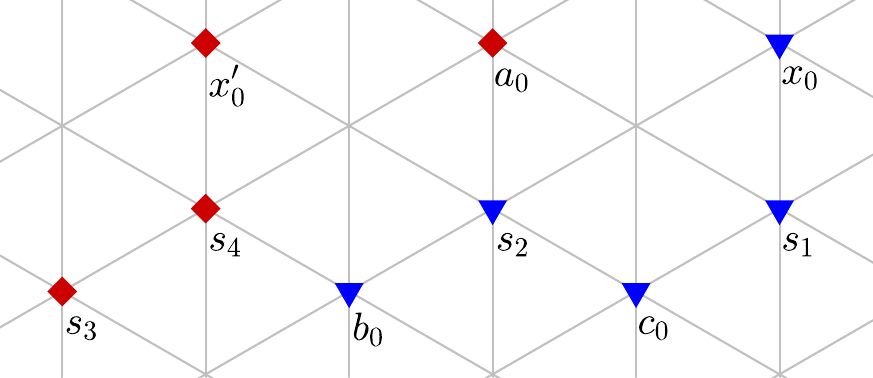}
        \caption{Reducing to the case that $x_0$ is red (just use $x_0'$ instead in the above)}
        \label{fig:lemma2graph2}
    \end{figure}
        \begin{align*}
        x_0, c_0 \implies s_1{}^*,
        \ \ \
        b_0, c_0 &\implies s_2{}^* \text{ \& } s_3, 
        \ \ \
        s_1,s_2 \implies s_4, 
        \ \ \ 
        b_0,s_3,s_4 \implies x_0'{}^*.
     \end{align*}
\end{proof}

To complete the proof of Theorem \ref{thm:ell3}, notice that if a point is colored red, all points at distance $\frac{4}{\sqrt{3}}$ in the grid are also red (Figure \ref{fig:coloredgrid}). By rotating the grid around this point, and around any sufficiently close blue point, we create two circles of radius $\frac{4}{\sqrt{3}}$ that must be entirely red and entirely blue that overlap, which is a contradiction. Therefore, the conditions of Lemmas \ref{lem:computationallemma} and \ref{lem:fixedgrid} must be false, i.e., $\E^2$ cannot be two-colored without a monochromatic $\ell_3$.
    \begin{figure}[H]
        \centering
        \includegraphics[scale = 0.35]{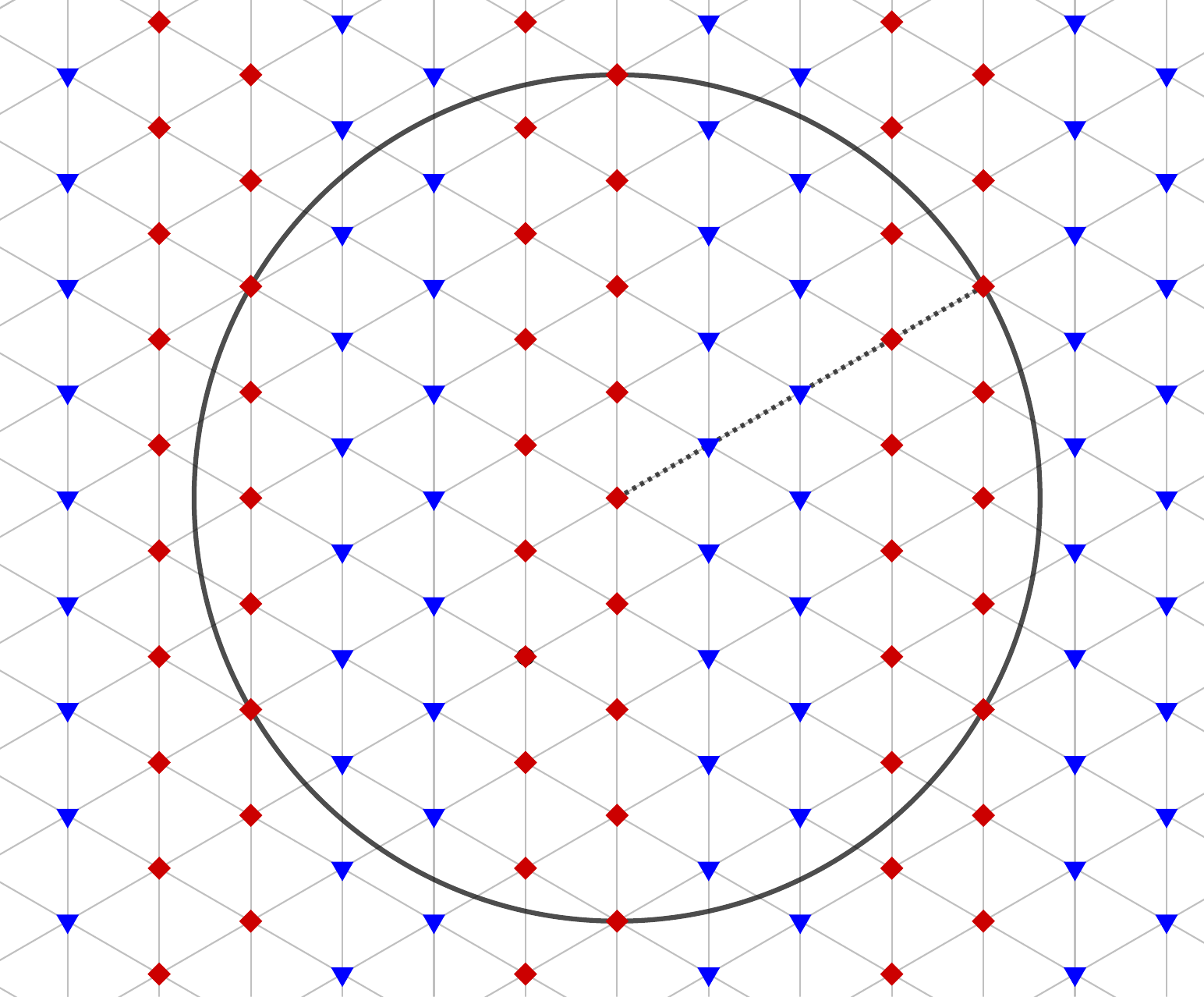}
        \caption{A circle with radius $\frac{4}{\sqrt{3}}$ in the colored grid}
        \label{fig:coloredgrid}
    \end{figure} 
    
\section*{Acknowledgements}
The authors thank J\'ozsef Solymosi and Joshua Zahl for helpful discussions. The third author also thanks Zixiang Xu for pointing out several relevant references. The authors thank the anonymous referees for their valuable comments and suggestions.

\bibliographystyle{abbrv}
\bibliography{references}

\begin{thebibliography}{10}

\bibitem{AT18}
A.~Arman and S.~Tsaturian.
\newblock A result in asymmetric {E}uclidean {R}amsey theory.
\newblock {\em Discrete Math.}, 341(5):1502--1508, 2018.

\bibitem{BMP}
P.~Brass, W.~Moser, and J.~Pach.
\newblock {\em Research problems in discrete geometry}.
\newblock Springer, New York, 2005.

\bibitem{CF19}
D.~Conlon and J.~Fox.
\newblock Lines in {E}uclidean {R}amsey theory.
\newblock {\em Discrete Comput. Geom.}, 61(1):218--225, 2019.

\bibitem{CW23}
D.~Conlon and Y.-H. Wu.
\newblock More on lines in {E}uclidean {R}amsey theory.
\newblock {\em C. R. Math. Acad. Sci. Paris}, 361:897--901, 2023.

\bibitem{CT94}
G.~Csizmadia and G.~T\'{o}th.
\newblock Note on a {R}amsey-type problem in geometry.
\newblock {\em J. Combin. Theory Ser. A}, 65(2):302--306, 1994.

\bibitem{Resources3AP}
G.~Currier, K.~Moore, and C.~H. Yip.
\newblock Resources for `{A}ny two-coloring of the plane contains monochromatic 3-term arithmetic progressions', 2024.
\newblock \url{https://personal.math.ubc.ca/~kjmoore/threeAPresources.html}.

\bibitem{dG18}
A.~D. N.~J. de~Grey.
\newblock The chromatic number of the plane is at least 5.
\newblock {\em Geombinatorics}, 28(1):18--31, 2018.

\bibitem{E73}
P.~Erd\H{o}s, R.~L. Graham, P.~Montgomery, B.~L. Rothschild, J.~Spencer, and E.~G. Straus.
\newblock Euclidean {R}amsey theorems. {I}.
\newblock {\em J. Combinatorial Theory Ser. A}, 14:341--363, 1973.

\bibitem{E75a}
P.~Erd\H{o}s, R.~L. Graham, P.~Montgomery, B.~L. Rothschild, J.~Spencer, and E.~G. Straus.
\newblock Euclidean {R}amsey theorems. {II}.
\newblock In {\em Infinite and finite sets ({C}olloq., {K}eszthely, 1973; dedicated to {P}. {E}rd\H{o}s on his 60th birthday), {V}ols. {I}, {II}, {III}}, volume Vol. 10 of {\em Colloq. Math. Soc. J\'{a}nos Bolyai}, pages 529--557. North-Holland, Amsterdam-London, 1975.

\bibitem{E75b}
P.~Erd\H{o}s, R.~L. Graham, P.~Montgomery, B.~L. Rothschild, J.~Spencer, and E.~G. Straus.
\newblock Euclidean {R}amsey theorems. {III}.
\newblock In {\em Infinite and finite sets ({C}olloq., {K}eszthely, 1973; dedicated to {P}. {E}rd\H{o}s on his 60th birthday), {V}ols. {I}, {II}, {III}}, volume Vol. 10 of {\em Colloq. Math. Soc. J\'{a}nos Bolyai}, pages 559--583. North-Holland, Amsterdam-London, 1975.

\bibitem{EI20}
G.~Exoo and D.~Ismailescu.
\newblock The chromatic number of the plane is at least 5: A new proof.
\newblock {\em Discrete \& Computational Geometry}, 64(1):216--226, 2020.

\bibitem{FT24}
J.~F\"uhrer and G.~T\'oth.
\newblock Progressions in {E}uclidean {R}amsey theory, 2024.
\newblock arXiv:2402.12567.

\bibitem{GOT18}
J.~E. Goodman, J.~O'Rourke, and C.~D. T\'{o}th, editors.
\newblock {\em Handbook of discrete and computational geometry}.
\newblock Discrete Mathematics and its Applications (Boca Raton). CRC Press, Boca Raton, FL, third edition, 2018.

\bibitem{GT11}
R.~Graham and E.~Tressler.
\newblock Open problems in {E}uclidean {R}amsey theory.
\newblock In {\em Ramsey theory}, volume 285 of {\em Progr. Math.}, pages 115--120. Birkh\"{a}user/Springer, New York, 2011.

\bibitem{JKSV09}
V.~Jel\'{\i}nek, J.~Kyn\v{c}l, R.~Stola\v{r}, and T.~Valla.
\newblock Monochromatic triangles in two-colored plane.
\newblock {\em Combinatorica}, 29(6):699--718, 2009.

\bibitem{J79}
R.~Juh\'{a}sz.
\newblock Ramsey type theorems in the plane.
\newblock {\em J. Combin. Theory Ser. A}, 27(2):152--160, 1979.

\bibitem{S76}
L.~E. Shader.
\newblock All right triangles are {R}amsey in {$E\sp{2}$}!
\newblock {\em J. Combinatorial Theory Ser. A}, 20(3):385--389, 1976.

\bibitem{S01}
A.~D. Szlam.
\newblock Monochromatic translates of configurations in the plane.
\newblock {\em J. Combin. Theory Ser. A}, 93(1):173--176, 2001.

\bibitem{T17}
S.~Tsaturian.
\newblock A {E}uclidean {R}amsey result in the plane.
\newblock {\em Electron. J. Combin.}, 24(4):Paper No. 4.35, 9, 2017.

\end{thebibliography}

\newpage
\appendix

\section{The remaining cases for Lemma \ref{lem:computationallemma}}
\label{app:lemmacases}
To recap the instructions to verify this proof, in each case we add several points to Figure \ref{fig:lemma1cases}, labeled $\set{s_1,...\, , s_{m}}$, and observe that $s_{m}$ cannot be colored red or blue. One should visit each vertex $s_k$ in the graphs, and note that its color is implied by vertices in $\set{p_1,...\, ,p_4}$, $\set{q_1,...\, ,q_5,q_3'}$, and $\set{s_1,...\, ,s_{k-1}}$. We have previously solved the case where $q_1$ is red and $q_2$ is blue, and by symmetry $q_1$ blue and $q_2$ red. To clarify things, the following list contains all possibilities. It may be the case that
\vspace{3mm}

\begin{enumerate}[itemsep=0.2cm]
    \item $q_1$ is red and $q_2$ is blue (case 1, done in Section \ref{sec:proof}).
    \item $q_1$ and $q_2$ are red (case 2).
    \item $q_1$ and $q_2$ are blue, and
    \vspace{2mm}
    
    \begin{enumerate}[itemsep=0.2cm]
        \item $q_3$ is red (case 3).
        \item $q_3$ is blue, and
        \vspace{2mm}
        
        \begin{enumerate}[itemsep=0.2cm]
            \item $q_4$ and $q_5$ are red (case 4).
            \item $q_4$ is red, $q_5$ is blue (case 5).
            \item $q_4$ and $q_5$ are blue (case 6).
        \end{enumerate}
    \end{enumerate}
\end{enumerate}    
\vspace{3mm}

    \noindent
      \textbf{Case 2:} $q_1$ and $q_2$ are red.
    \begin{figure}[H]
        \centering
        \includegraphics[scale = 0.58, trim = 0 0 0 20]{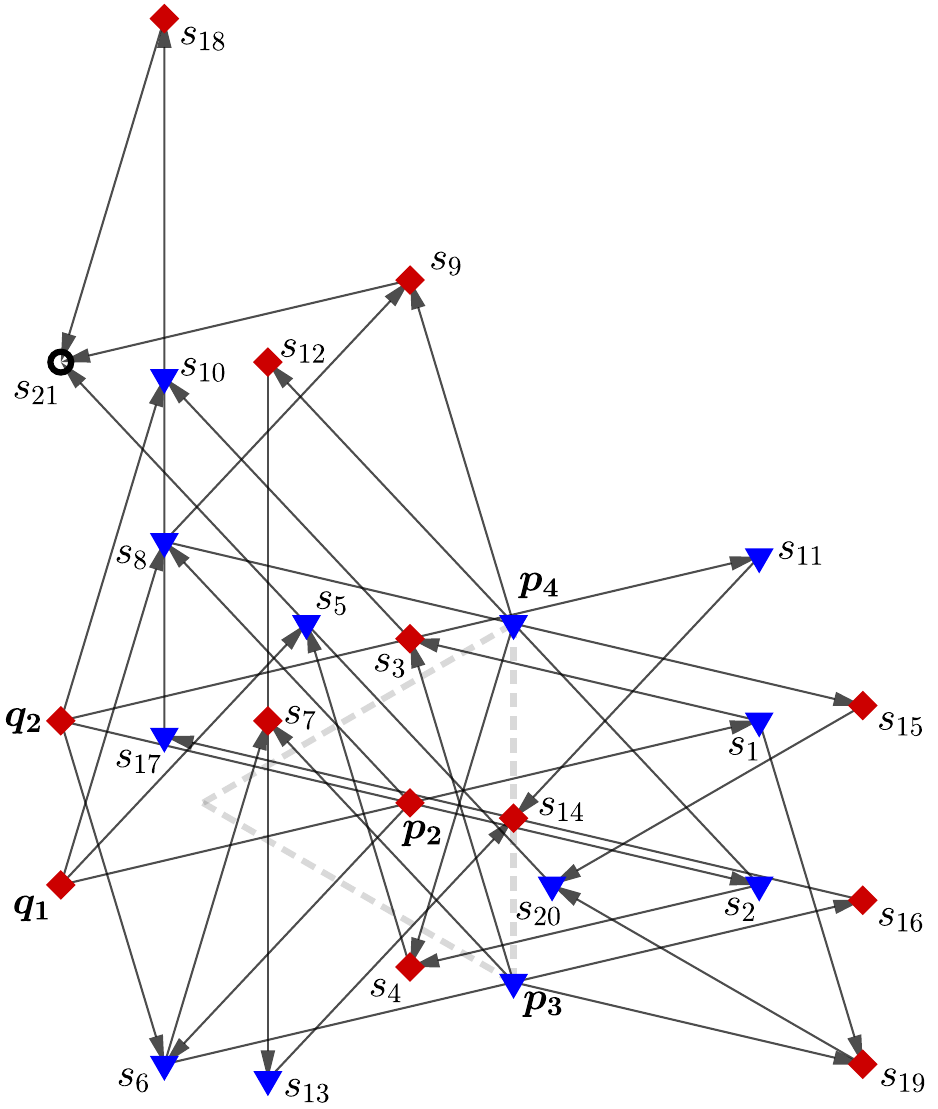}
        \caption{Case 2: $q_1$ and $q_2$ are red}
        \label{fig:lemma1case2}
    \end{figure}

    \noindent
    \textbf{Case 3:} $q_1$ and $q_2$ are blue, and $q_3$ (or $q_3'$) is red.
    \begin{figure}[H]
        \centering
        \includegraphics[scale = 0.55, trim = 0 0 0 14]{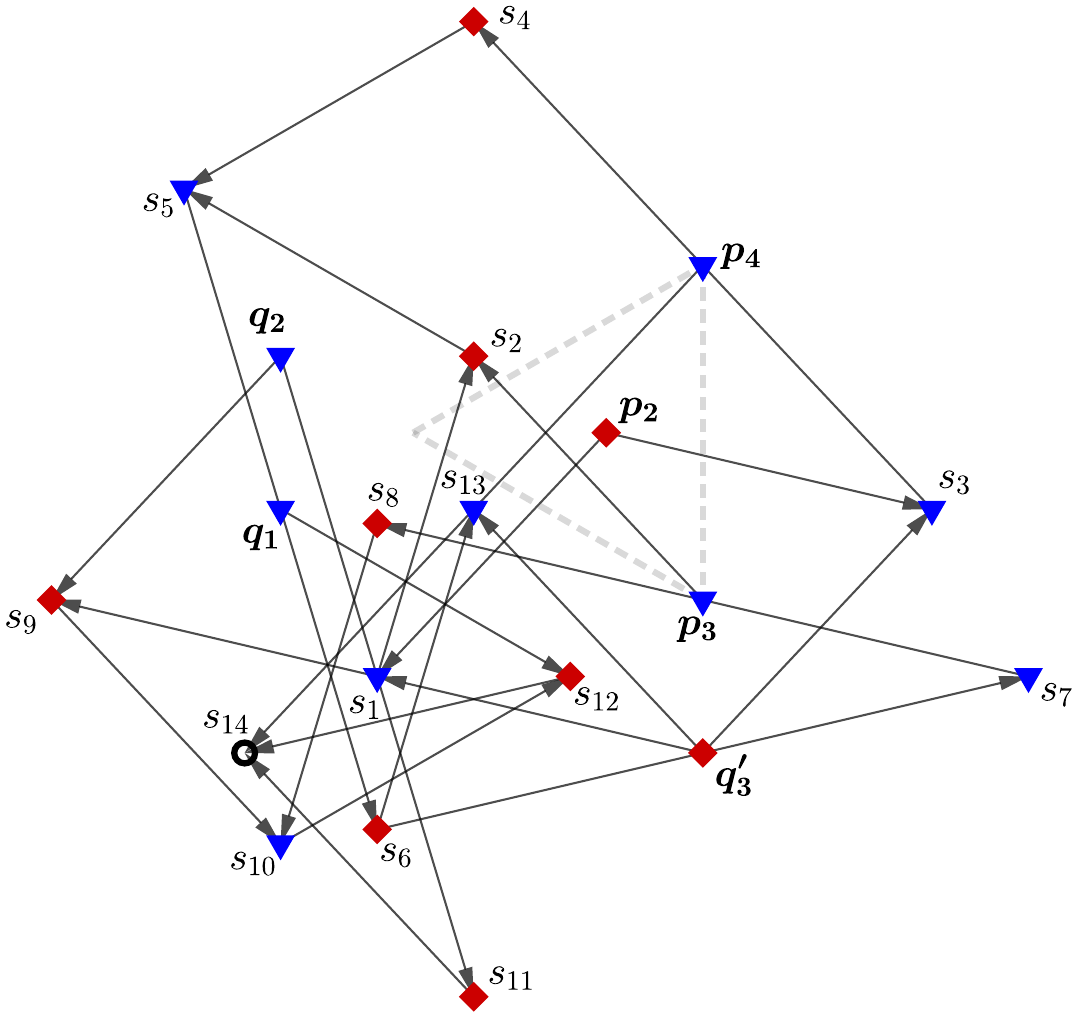}
        \caption{Case 3: $q_3$ (or $q_3'$) is red}
        \label{fig:lemma1case3}
    \end{figure}

    \noindent
    \textbf{Case 4:} $q_1$, $q_2$ $q_3$, and $q_3'$ are blue, $q_4$ and $q_5$ are red.
    \begin{figure}[H]
        \centering
        \includegraphics[scale = 0.585, trim = 0 0 0 14]{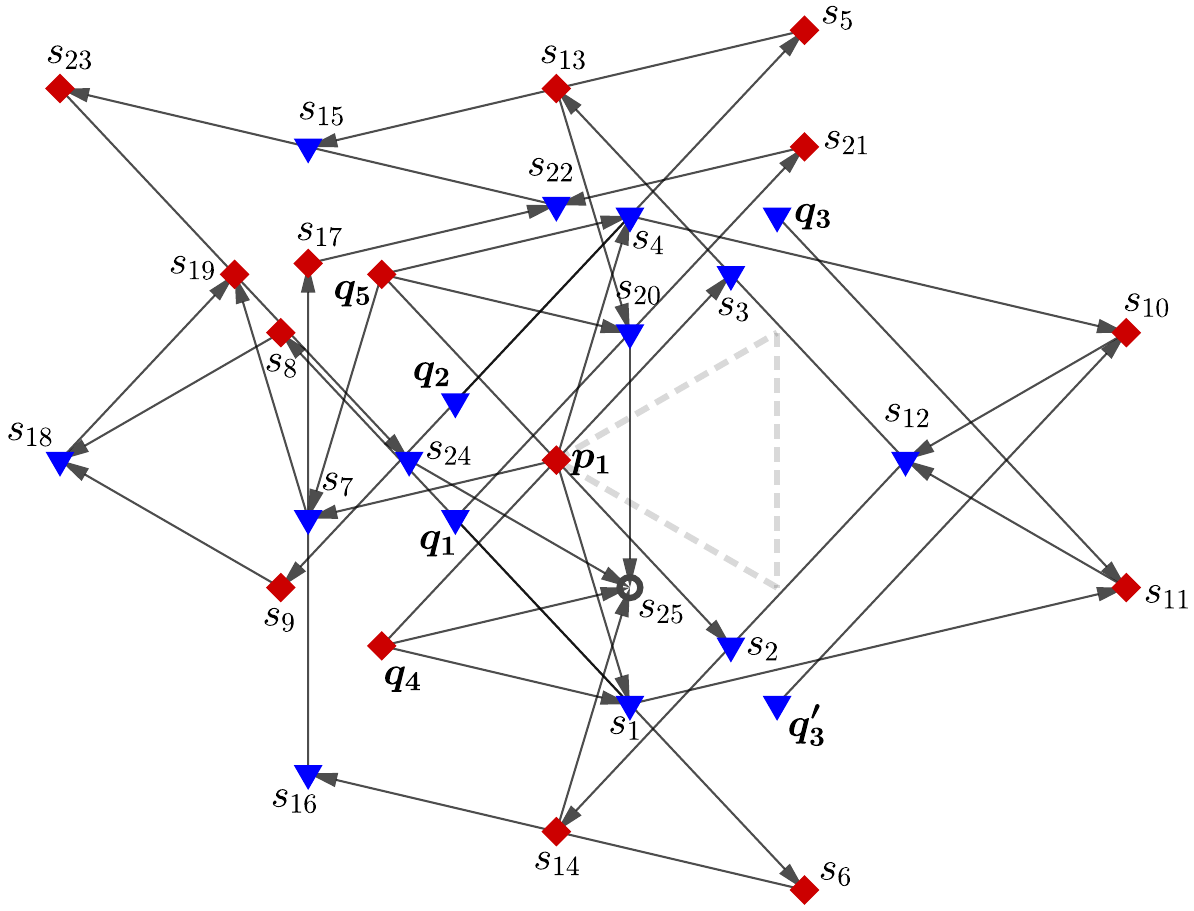}
        \caption{Case 4: $q_4$ and $q_5$ are red}
        \label{fig:lemma1case4}
    \end{figure}
    \vspace{3mm}

    \noindent
    \textbf{Case 5:} $q_1$, $q_2$ $q_3$, $q_3'$, and $q_5$ are blue, $q_4$ is red, 
   \begin{figure}[H]
        \centering
        \includegraphics[scale = 0.52, trim = 0 14 0 14]{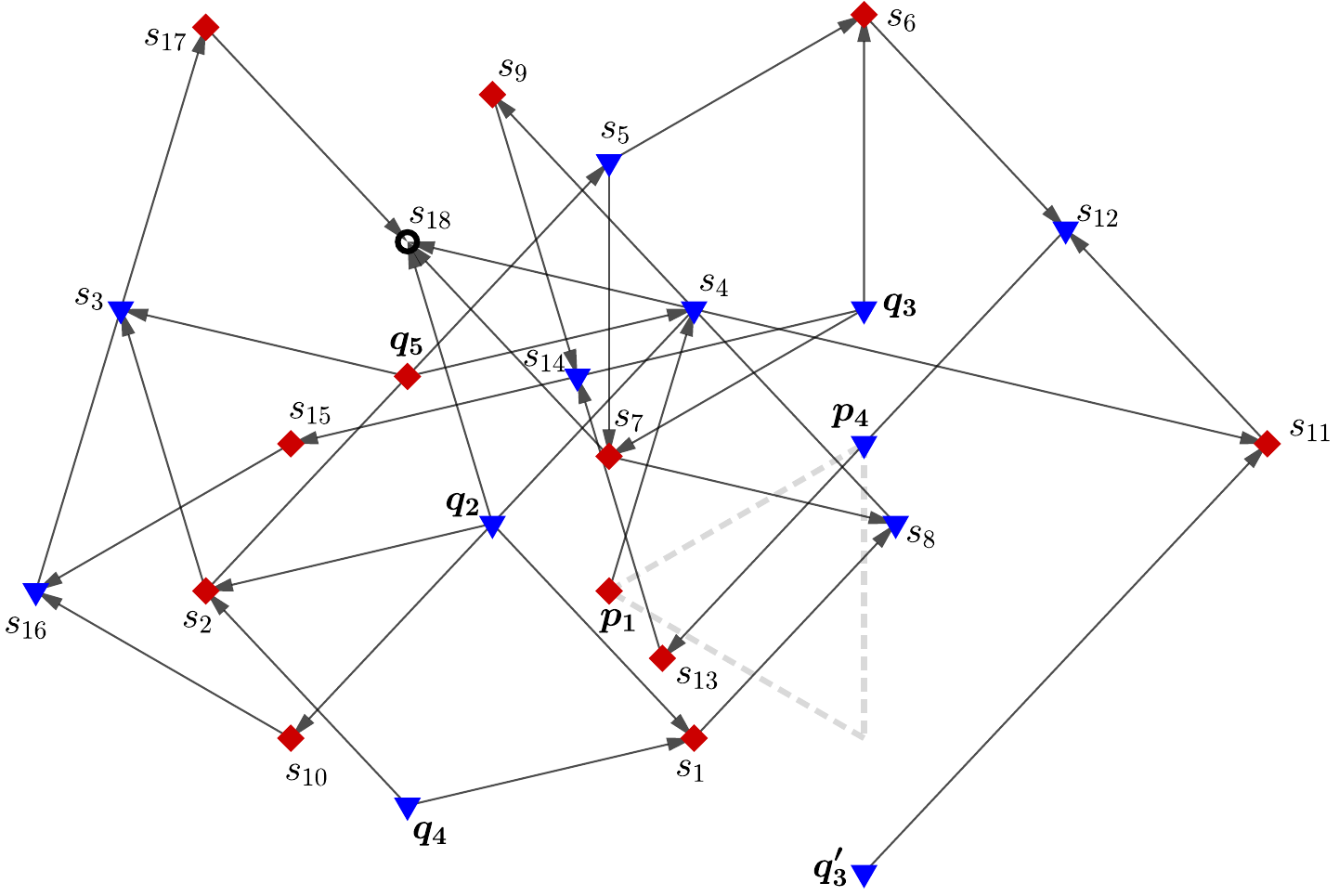}
        \caption{Case 5: $q_4$ is red, $q_5$ is blue}
        \label{fig:lemma1case5}
    \end{figure}

    \textbf{Case 6:} All of $q_1$, $q_2$ $q_3$, $q_3'$, $q_4$, and $q_5$ are blue.

   \begin{figure}[H]
        \centering
        \includegraphics[scale = 0.52, trim = 0 14 0 14]{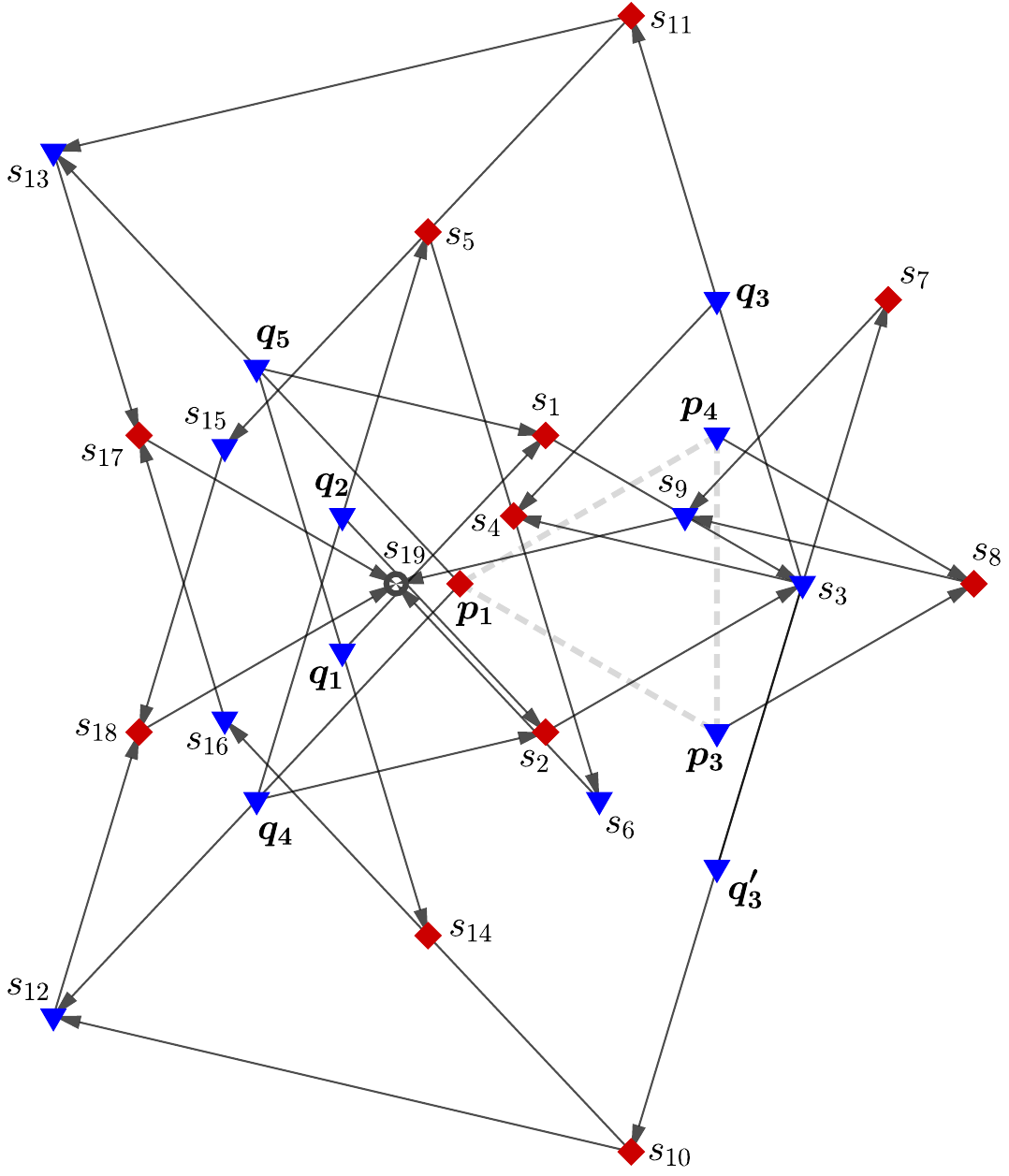}
        \caption{Case 6: $q_4$ and $q_5$ are blue}
        \label{fig:lemma1case6}
    \end{figure}

\section{Point set coordinates}
    \label{app:pointsets}

\textbf{Figure \ref{fig:lemma1graph}: 56 point contradiction} 
\begin{align*}
&[-4, 0, 0, 0]\ p_1,\ 
[0, 0, 0, 0]\ p_2,\ 
[2, 0, -6, 0]\ p_3,\ 
[2, 0, 6, 0]\ p_4,\ 
[-1, -3, 3, -1]\ q_1,\ 
\\
&[-1, -3, -3, 1]\ q_2,\ 
[2, 0, 0, 2]\ q_3,\ 
[2, 0, 0, -2]\ q_3',v
[-3, -3, -3, -1]\ q_4,\ 
[-3, -3, 3, 1]\ q_5,\ 
\\
&[-1, -3, -3, -3],\ 
[-1, 3, 3, 1],\ 
[-1, 3, -9, 1],\ 
[3, 3, -3, -1],\ 
[-2, 0, -6, 0],\ 
[1, 3, -9, -1],\ 
\\
&[-2, 0, 0, -2],\ 
[1, -3, -3, -1],\ 
[4, 0, 0, 0],\ 
[-1, 3, 9, -1],\ 
[0, 0, -6, -2],\ 
[-1, -3, 9, 1],\ 
\\
&[1, 3, -15, 1],\ 
[1, -3, 3, 1],\ 
[1, 3, -3, 1],\ 
[-2, 0, 6, 0],\ 
[-3, -3, 9, -1],\ 
[1, 3, -3, -3],\ 
\\
&[6, 0, 0, 2],\ 
[5, -3, -3, -1],\ 
[-1, -3, 3, 3],\ 
[3, 3, -9, 1],\ 
[-5, 3, 3, 1],\ 
[6, 0, 0, -2],\ 
\\
&[-2, 6, 0, 0],\ 
[1, 3, 3, -1],\ 
[5, -3, 3, 1],\ 
[-1, -3, -15, 1],\ 
[0, 0, 6, 2],\ 
[3, -3, -3, 1],\ 
\\
&[0, 0, 6, -2],\ 
[5, 3, -3, 1],\ 
[-1, 3, -3, -1],\ 
[0, 0, -6, 2],\ 
[0, 0, -12, 0],\ 
[-3, 3, 3, -1],\ 
\\
&[-1, -3, -9, -1],\ 
[3, -3, 3, -1],\ 
[-2, 0, 0, 2],\ 
[-1, 3, 3, -3],\ 
[-3, 3, -3, 1],\ 
[3, 3, 3, 1],\ 
\\
&[5, 3, 3, -1],\ 
[1, -3, 3, -3],\ 
[1, 3, 9, 1],\ 
[1, 3, 3, 3].
\end{align*}

\
    
\textbf{Figure \ref{fig:lemma1case1}: Lemma \ref{lem:computationallemma} Case 1} 
\begin{align*}
&[0, 0, 0, 0]\ p_2,\ 
[2, 0, -6, 0]\ p_3,\ 
[2, 0, 6, 0]\ p_4,\ 
[-1, -3, 3, -1]\ q_1,\ 
[-1, -3, -3, 1]\ q_2,\ 
\\
&[1, -3, -3, -1]\ s_{1},\ 
[1, 3, 3, -1]\ s_{2},\ 
[3, 3, 3, 1]\ s_{3},\ 
[0, 0, 6, -2]\ s_{4},\ 
[3, -3, -3, 1]\ s_{5},\ 
\\
&[3, -3, 9, 1]\ s_{6},\ 
[0, 0, -6, -2]\ s_{7},\ 
[-3, -3, 3, 1]\ s_{8},\ 
[-2, -6, 0, 0]\ s_{9},\ 
[0, -6, 0, 2]\ s_{10},\ 
\\
&[1, -3, 3, 1]\ s_{11},\ 
[2, 0, 0, 2]\ s_{12},\ 
[1, -3, 3, -3]\ s_{13},\ 
[-1, 3, 9, -1]\ s_{14},\ 
[6, 0, 0, -2]\ s_{15},\ 
\\
&[4, 0, 0, 0]\ s_{16},\ 
[5, 3, 3, -1]\ s_{17},\ 
[-1, 3, -3, -1]\ s_{18}.
\end{align*}

\
    
\textbf{Figure \ref{fig:lemma1case2}: Lemma \ref{lem:computationallemma} Case 2} 
\begin{align*}
&[0, 0, 0, 0]\ p_2,\ 
[2, 0, -6, 0]\ p_3,\ 
[2, 0, 6, 0]\ p_4,\ 
[-1, -3, 3, -1]\ q_1,\ 
[-1, -3, -3, 1]\ q_2,\ 
\\
&[1, 3, -3, 1]\ s_{1},\ 
[1, 3, 3, -1]\ s_{2},\ 
[0, 0, -6, 2]\ s_{3},\ 
[0, 0, 6, -2]\ s_{4},\ 
[-2, 0, 6, 0]\ s_{5},\ 
\\
&[1, -3, -3, -1]\ s_{6},\ 
[3, -3, -3, 1]\ s_{7},\ 
[1, -3, 3, 1]\ s_{8},\ 
[0, 0, 6, 2]\ s_{9},\ 
[1, -3, -3, 3]\ s_{10},\ 
\\
&[1, 3, -9, 3]\ s_{11},\ 
[3, -3, 9, 1]\ s_{12},\ 
[3, -3, -15, 1]\ s_{13},\ 
[2, 0, -12, 2]\ s_{14},\ 
[3, 3, 9, -1]\ s_{15},\ 
\\
&[3, 3, -9, 1]\ s_{16},\ 
[1, -3, -15, 3]\ s_{17},\ 
[1, -3, 9, 3]\ s_{18},\ 
[3, 3, -3, -1]\ s_{19},\ 
[-3, 3, 3, -1]\ s_{20},\ 
\\
&[-1, -3, 9, 1]\ s_{21}.
\end{align*}

\
    
\textbf{Figure \ref{fig:lemma1case3}: Lemma \ref{lem:computationallemma} Case 3} 
\begin{align*}
& [0, 0, 0, 0]\ p_2,\ 
[2, 0, -6, 0]\ p_3,\ 
[2, 0, 6, 0]\ p_4,\ 
[-1, -3, 3, -1]\ q_1,\ 
[-1, -3, -3, 1]\ q_2,\ 
\\
&[2, 0, 0, -2]\ q_3',\ 
[1, -3, -3, -1]\ s_{1},\ 
[3, -3, -3, 1]\ s_{2},\ 
[1, 3, 3, -1]\ s_{3},\ 
[3, -3, 9, 1]\ s_{4},\ 
\\
&[-3, -3, 3, 1]\ s_{5},\ 
[1, -3, 3, -3]\ s_{6},\ 
[3, 3, -3, -1]\ s_{7},\ 
[1, -3, -9, 1]\ s_{8},\ 
[0, -6, -6, 0]\ s_{9},\ 
\\
&[-1, -3, -9, -1]\ s_{10},\ 
[3, -3, -3, -3]\ s_{11},\ 
[5, -3, -3, -1]\ s_{12},\ 
[3, -3, 3, -1]\ s_{13},\ 
\\
&[4, -6, 0, -2]\ s_{14}. 
\end{align*}

\
    
\textbf{Figure \ref{fig:lemma1case4}: Lemma \ref{lem:computationallemma} Case 4} 
\begin{align*}
&[-4, 0, 0, 0]\ p_1,\ 
[-1, -3, 3, -1]\ q_1,\ 
[-1, -3, -3, 1]\ q_2,\ 
[2, 0, 0, 2]\ q_3,\ 
[2, 0, 0, -2]\ q_3',\ 
\\
&[-3, -3, -3, -1]\ q_4,\ 
[-3, -3, 3, 1]\ q_5,\ 
[-2, 0, 0, -2]\ s_1,\ 
[-5, 3, -3, -1]\ s_2,\ 
[-5, 3, 3, 1]\ s_3,\ 
\\
&[-2, 0, 0, 2]\ s_4,\ 
[-3, 3, 3, 3]\ s_5,\ 
[-3, 3, -3, -3]\ s_6,\ 
[-5, -3, 3, -1]\ s_7,\ 
[0, -6, 6, 0]\ s_8,\ 
\\
&[0, -6, -6, 0]\ s_9,\ 
[0, 6, 6, 0]\ s_{10},\ 
[0, 6, -6, 0]\ s_{11},\ 
[-6, 6, 0, 0]\ s_{12},\ 
[-4, 0, 6, 2]\ s_{13},\ 
\\
&[-4, 0, -6, -2]\ s_{14},\ 
[-5, -3, 9, 1]\ s_{15},\ 
[-5, -3, -9, -1]\ s_{16},\ 
[-5, -3, 15, -1]\ s_{17},\ 
\\
&[-6, -6, 0, 0]\ s_{18},\ 
[-7, -3, 3, 1]\ s_{19},\ 
[-2, 0, 6, 0]\ s_{20},\ 
[-3, 3, 9, 1]\ s_{21},\ 
[-4, 0, 12, 0]\ s_{22},\ 
\\
&[-6, -6, 6, 2]\ s_{23},\ 
[-8, 0, 0, 0]\ s_{24},\ 
[-2, 0, -6, 0]\ s_{25}.
\end{align*}

\
    
\textbf{Figure \ref{fig:lemma1case5}: Lemma \ref{lem:computationallemma} Case 5} 
\begin{align*}
&[-4, 0, 0, 0]\ p_1,\ 
[2, 0, 6, 0]\ p_4,\ 
[-1, -3, -3, 1]\ q_2,\ 
[2, 0, 0, 2]\ q_3,\ 
[2, 0, 0, -2]\ q_3',\ 
\\
&[-3, -3, -3, -1]\ q_4,\ 
[-3, -3, 3, 1]\ q_5,\ 
[-2, 0, -6, 0]\ s_{1},\ 
[-2, -6, 0, 0]\ s_{2},\ 
[-4, -6, 0, 2]\ s_{3},\ 
\\
&[-2, 0, 0, 2]\ s_{4},\ 
[-4, 0, 6, 2]\ s_{5},\ 
[2, 0, 12, 2]\ s_{6},\ 
[-4, 0, -6, 2]\ s_{7},\ 
[-3, 3, -3, 1]\ s_{8},\ 
\\
&[-1, -3, 3, 3]\ s_{9},\ 
[0, -6, -6, 0]\ s_{10},\ 
[0, 6, 6, 0]\ s_{11},\ 
[1, 3, 9, 1]\ s_{12},\ 
[3, -3, 3, -1]\ s_{13},\ 
\\
&[1, -3, 3, 1]\ s_{14},\ 
[0, -6, 6, 0]\ s_{15},\ 
[-6, -6, 0, 0]\ s_{16},\ 
[-2, -6, 0, 4]\ s_{17},\ 
[-3, -3, -3, 3]\ s_{18}. 
\end{align*}

\
    
\textbf{Figure \ref{fig:lemma1case6}: Lemma \ref{lem:computationallemma} Case 6} 
\begin{align*}
&[-4, 0, 0, 0]\ p_1,\ 
[2, 0, -6, 0]\ p_3,\ 
[2, 0, 6, 0]\ p_4,\ 
[-1, -3, 3, -1]\ q_1,\ 
[-1, -3, -3, 1]\ q_2,\ 
\\
&[2, 0, 0, 2]\ q_3,\ 
[2, 0, 0, -2]\ q_3',\ 
[-3, -3, -3, -1]\ q_4,\ 
[-3, -3, 3, 1]\ q_5,\ 
[-2, 0, 6, 0]\ s_{1},\ 
\\
&[-2, 0, -6, 0]\ s_{2},\ 
[4, 0, 0, 0]\ s_{3},\ 
[3, -3, -3, 1]\ s_{4},\ 
[1, -3, -3, 3]\ s_{5},\ 
[5, -3, -3, -1]\ s_{6},\ 
\\
&[6, 0, 0, 2]\ s_{7},\ 
[8, 0, 0, 0]\ s_{8},\ 
[7, -3, -3, 1]\ s_{9},\ 
[0, 0, 0, -4]\ s_{10},\ 
[0, 0, 0, 4]\ s_{11},\ 
\\
&[-2, -6, -6, -2]\ s_{12},\ 
[-2, -6, 6, 2]\ s_{13},\ 
[1, -3, 3, -3]\ s_{14},\ 
[2, -6, -6, 2]\ s_{15},\ 
\\
&[2, -6, 6, -2]\ s_{16},\ 
[0, -6, 6, 0]\ s_{17},\ 
[0, -6, -6, 0]\ s_{18},\ 
[6, -6, 0, 0]\ s_{19}.
\end{align*}
\end{document}